\documentclass[12pt]{amsart}

\usepackage{amssymb,amsmath,psfrag}
\usepackage{graphics,graphicx}

 \markleft{\sc R. Fl\'orez}

\newtheorem{lem}{Lemma}[section]   % USE THIS version if you want to number by section.
\newtheorem{cor}[lem]{Corollary}
\newtheorem{prop}[lem]{Proposition}
\newtheorem{thm}[lem]{Theorem}

\theoremstyle{remark}

\newcommand\cl{\operatorname{cl}}
\newcommand\rk{\operatorname{r}}

\newcommand\HP{\mathrm{HP}}
\newcommand\Hc{\mathrm{H}}

\newcommand\cL{\mathcal L}

\begin{document}

\begin{center}

\Large{\bf Harmonic Conjugation in Harmonic Matroids}
\normalsize

{\sc Rigoberto Fl\'orez} \footnote {Part of the work was performed
at the State University of New York at Binghamton and formed part of
the author's doctoral dissertation.}

{\sc University of South Carolina Sumter\\
     Sumter, SC, U.S.A.\ 29150-2498}

\footnotesize

% {\tt florezr@uscsumter.edu}
\normalsize

%Revised version of \today.

\end{center}

\bigskip

\footnotesize {\it Abstract:} {We study a generalization of the
concept of harmonic conjugation from projective geometry and full
algebraic matroids to a larger class of matroids called \emph
{harmonic matroids}. We use harmonic conjugation to construct a
projective plane of prime order in harmonic matroids without using
the axioms of projective geometry. As a particular case we have a
combinatorial construction of a projective plane of prime order in
full algebraic matroids.}

\normalsize

\thispagestyle{empty}

\vspace{0.2cm}
\section {Introduction}
Lindstr\"om  \cite{lhc} generalized the concept of harmonic
conjugation from projective geometry to full algebraic matroids. He
used this concept to construct many algebraically nonrepresentable
matroids. The method of harmonic conjugation generalizes to many
other matroids. Here we prove that some basic properties of harmonic
conjugation in projective geometry can be generalized to a larger
class of matroids called harmonic matroids of which the full
algebraic matroids are an example. Also, we show that harmonic
conjugation gives a geometric construction of a projective plane of
prime order in any harmonic matroid without using the axioms of
projective geometry.

We show in section \ref{harmonicconjugation} that our generalized
harmonic conjugacy has the usual symmetries. In the main theorem
(Theorem \ref {PP}) we show that if a small matroid called $L_p$ is
embeddable in a harmonic matroid $M$, then $L_p$ can be extended
using only harmonic conjugation to a projective plane of order $p$
prime, in $M$. As a corollary, this construction holds in full
algebraic matroids. In the section \ref{minimalmatroid} we show that
a smaller submatroid $R_{\text {cycle}} [p]$ of  $L_p$ is a minimal
matroid that extends by harmonic conjugation to a projective plane
of order $p$ prime. In the last section we generalize the concept of
harmonic sequence from projective geometry to harmonic matroids. We
prove that this sequence gives rise to a finite field and a M\"obius
harmonic net in a harmonic matroid.

A question raised by Lindstr\"om \cite {lam} asks which projective
geometries can be embedded in a full algebraic matroid. He gave a
partial answer to this question in \cite {lpg}, constructing some
examples. Evans and Hrushovski \cite {eh} gave a complete answer
using model theory and the theory of algebraic groups. It is still
open how to solve the question using algebraic or geometric
techniques. In this paper I make a contribution to the solution.

\section {Preliminaries}

We denote the closure operator and the rank function of matroids by
$\cl$ and $\rk$, respectively.

For elements $y, x, z, o, q, r, s$ of a matroid $M$, define $\HP(y,
x, z; o, q, r, s)$ to mean that the restriction $ M \mid \{y, x, z,
o, q, r, s \}$ is either of the matroids in Figure \ref{f1}. For
elements $y,x,z$ of $M$, define $x'$ to be a \emph{harmonic
conjugate} of $x$ with respect to $y$ and $z$ denoted
$\Hc(y,z;x,x')$, to mean that there are elements $y, x, z, o, q, r,
s$ with $\HP(y, x, z; o, q, r, s)$ for which $\cl (\{y,z \}) \cap
\cl (\{o,r\})$ is $x'$. A \emph{harmonic matroid} is a matroid such
that each set $\{y, x, z, o, q, r, s\}$ with $\HP(y, x, z; o, q, r,
s)$ gives rise to a harmonic conjugate $x'$ of $x$ with respect to
$y$ and $z$, and, furthermore, $x'$ depends only on $x,y,$ and $z$,
and not on the choice of $o,q, r$ and $s$.

\begin{figure} [htbp]
 \centering
\includegraphics[scale=0.5]{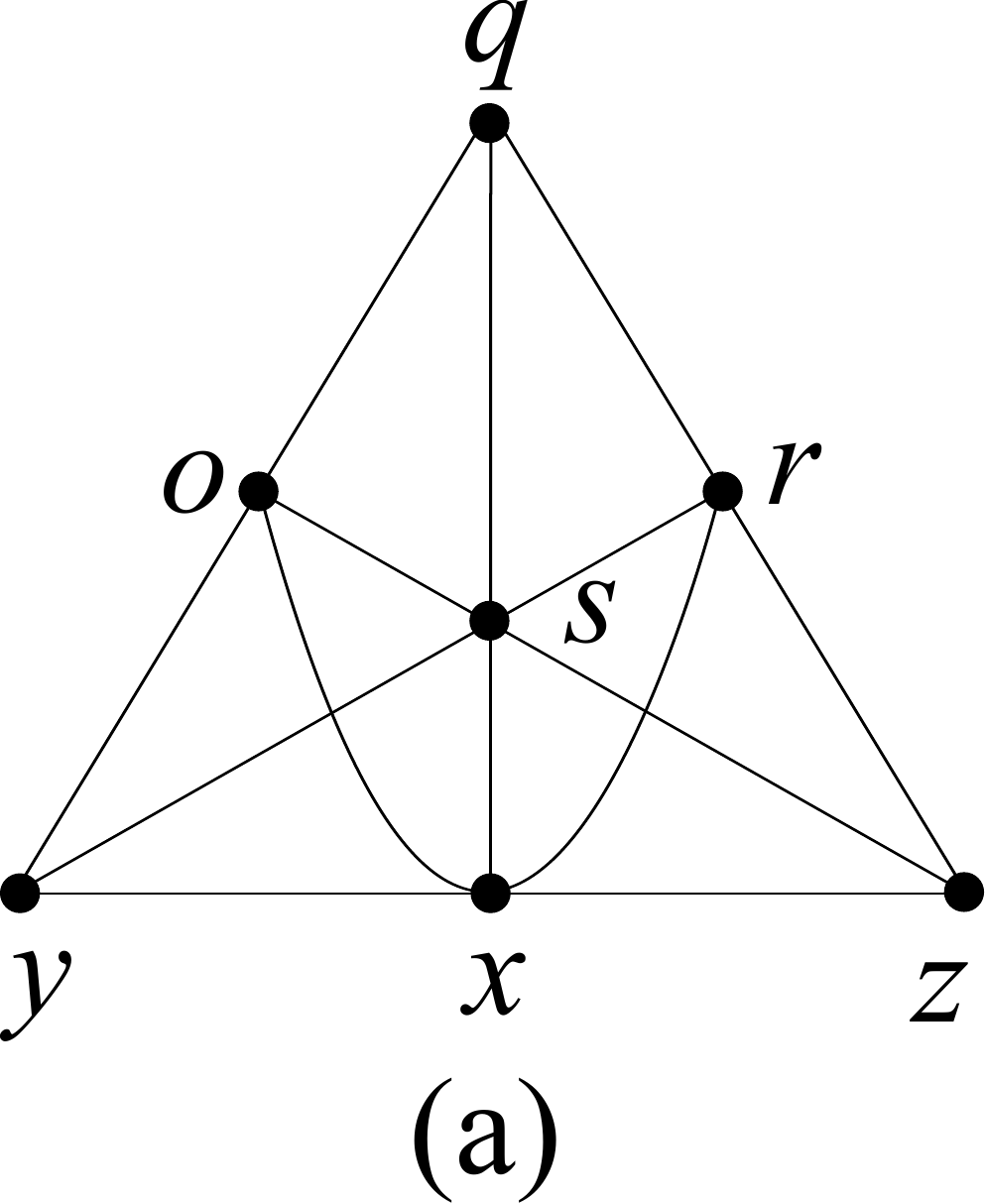} \hspace{2.5cm} \includegraphics[scale=0.5]{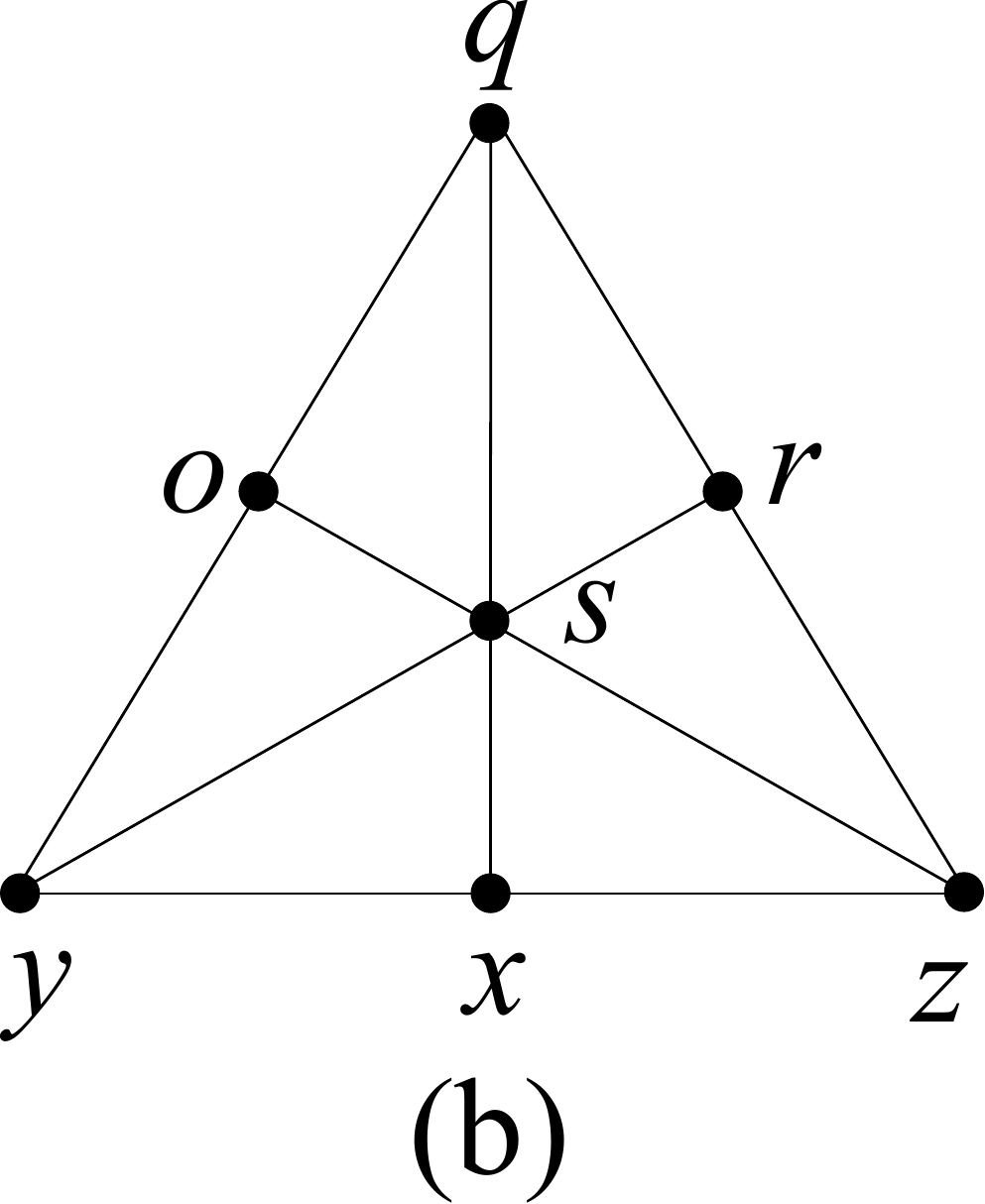}
\caption{$(a)$ The Fano matroid $F_7$. \ \ $(b)$ The non-Fano matroid $F_7^-.$} \label{f1}
\end{figure}

Familiar examples of harmonic matroids include Desarguesian planes
and all projective geometries of rank four or more. Also, a
\emph{little Desargues plane} is a harmonic matroid; this is a
projective plane whose ternary ring is an alternative ring
\cite{st}. Such a plane is also called  an \emph{alternative} or
\emph{Moufang plane} \cite [Corollary 14.2.5] {st}. In projective
planes the definition of harmonic conjugate does not require
uniqueness; it is known that, given three collinear points in any
projective plane, a harmonic conjugate of one point with respect to
the other two points always exists, but it is not necessarily
unique. For a projective plane having no $F_7$ restriction, all
harmonic conjugates in the plane are unique if and only if the plane
is a little Desargues plane {\cite [Theorems 5.7.4 and 5.7.12] {st}.
Thus, a projective plane having no $F_7$ restriction is a harmonic
matroid if and only if it is a little Desargues plane.

Full algebraic matroids are harmonic, as Lindstr\"om \cite{lhc}
proved. Let $F$ and $K$ be two fields such that $F \subseteq K.$
Assume that $F$ and $K$ are algebraically closed and $K$ has finite
transcendence degree over $F$. Then those subfields of $K$ which are
algebraically closed and contain $F$ form a lattice. The lattice of
such subfields, ordered by inclusion, is isomorphic to the lattice
of flats of a matroid. We call this matroid the \emph{full algebraic
matroid} $A(K/F)$. A set $S \subseteq K$ is independent in this
matroid if and only if its elements are algebraically independent
over $F$. A full algebraic matroid $A(K/F)$ of rank 3 consists of
the field $F$ (the flat of transcendence degree zero over $F$),
atoms or points (flats of transcendence degree one over $F$), lines
(flats of transcendence degree two) and a plane, which is the field
$K$ (since $K$ is of transcendence degree 3 over $F$).

An \emph{algebraic representation} (or \emph{embedding}) over a field $F$ of a
matroid $M$ is a mapping $f$ from the elements of $M$ into an
extension field $K$ of $F$, such that a subset $S$ is independent in
$M$ if and only if $f$ is injective and $f(S)$ is algebraically
independent over $F$.

We use a natural notation for points and lines given by the
projective  plane $PG(2,p)$ where $p$ is a prime (following
\cite{st}). So, we use homogeneous coordinates. For the points we
write $[x,y,z]$ and we have $ [x,y,z] = [\gamma x,\gamma y,\gamma
z]$ whenever $ \gamma \not = 0$. For the lines we define
$$\langle a,b,c \rangle_p =: \{ [x,y,z] : ax + by + cz = 0 \}$$
and we have $\langle a,b,c \rangle_p = \langle \gamma a,\gamma
b,\gamma c \rangle_p$ whenever $\gamma \not = 0$. Here $a, b, c, x,
y ,z \in \mathbb Z_p$ and not all of $a,b,c$ or $x,y,z$ can be zero.

Denote by $L_p$ the restriction of $PG(2,p)$, to three concurrent lines. An explicit
description of this matroid is as follows. The ground set is
$$ E := \big \{ [0,i,1],[1,i,1], [1,i,0]: i \in \mathbb Z_p \big \} \cup \big \{ [0,1,0]  \big \}.$$
The dependent lines are given by:

$\langle 1, 0, 0 \rangle_p := \big\{ [ 0,i,1]: i \in {\mathbb Z_p}
\big\} \cup \big\{ [0,1,0] \big\}$,

$ \langle 1, 0, -1 \rangle_p :=  \big\{ [ 1,i,1]: i \in {\mathbb Z_p}
\big\} \cup \big\{[0,1,0] \big\}$,

$ \langle  0, 0, 1 \rangle_p :=  \big\{ [1,i,0]: i \in {\mathbb Z_p}
\big\} \cup \big\{ [0,1,0] \big\}$,

$ \langle j,-1, i \rangle _p :=  \big\{  [0,i,1], [1,i+j,1], [1,j,0]
\big\} $ for $ i,j \in {\mathbb Z_p}$.

\section {Properties of harmonic conjugation} \label{harmonicconjugation}

In this section we generalize to harmonic matroids some elementary
properties of the harmonic conjugate that hold in a projective
geometry. (Lindstr\"om \cite {lhc} proved Proposition
\ref{conjugates} (\emph{i}) for full algebraic matroids.)

\begin{lem} \label{tconjugate}
If $ \HP(y, x, z; o, q, r, s) $ and $\Hc(y,z;x,x')$ in a harmonic
matroid, then $\Hc(y,r;s,t)$ for some point $t$ that is collinear
with $ q $ and $ x'$.
\end{lem}

\begin{proof}
Since $\Hc(y,z;x,x')$, the points $o, r, x'$ are collinear. Thus,
$\HP(y, s, r; x', z, q, o)$, so there is a point $t$ with
$\Hc(y,r;s,t)$. Therefore $\{ y, s, r, t \}$ and $ \{ q, t, x' \}$
are collinear sets.
\end{proof}

\begin{prop} \label{conjugates} In a harmonic matroid,
\begin{enumerate}
\item [] (i) $\Hc(y,z ; x,x')$ if and only if $\Hc(y,z ; x', x )$,
\item [] (ii) if $x \not = x'$ and $y \not = z$ then  $\Hc(y,z;x,x')$  if and only if $\Hc(x,x'; y,z)$.
\end{enumerate}
\end{prop}

\begin{proof} If $\Hc(y,z ; x,x')$, then there are
points $o,q,r,s$  with $ \HP(y, x, z; o, q, r, s) $ and  with $o ,r,
x'$ collinear. Thus, $ \HP(y, x', z; q, o, s, r) $, so there is a
point $x''$ with $H(y,z;x',x'')$. The point $x''$ being the
intersection of the lines $\cl(\{q,s \})$ and $\cl (\{y,z \})$, is
$x$, so $\Hc(y,z ; x',x)$, proving part $(i)$.

If $\Hc(y,z;x,x')$, then there are points $o,q,r,s $ with $\HP(y, x,
z; o, q, r, s) $ and  with $o ,r, x'$ collinear. Furthermore by
Lemma \ref{tconjugate} there is a point $t$ with $\{y,s,r,t \}$  and
$\{q,x',t \}$ collinear sets. Thus, $\HP(y, q, o; x', z, s, r)$, so
there is a point $q'$ with $\Hc(y,o;q,q')$ and with $x',s,q'$
collinear. Therefore, $ \HP(q, t, x'; o, y, z, r ) $ and  $ \HP(q,
t, x'; q', y, x, s) $, these and the uniqueness of the harmonic
conjugate imply that there is a point $w$ with $\Hc(q,x';t,w)$ and
with $ \{ q', x, w \}$ and  $\{s, o, z, w\} $ collinear. Thus, $
\HP(x, y, x'; s, q, w, q') $, so there is a point $y'$ with
$\Hc(x,x';y,y')$. The point $y'$, being the intersection of the
lines $\cl(\{o, s \})$ and $\cl (\{x,x' \})$, is $z$, so
$\Hc(x,x';y,z)$, proving part $(ii)$.
\end{proof}

\begin{lem} \label{tecnico}
If $\HP (y, x, z; o, q, r, s)$ and $\Hc(y,z ; x,x')$ in a harmonic
matroid, then

$(i)$ $\Hc(y,o ; q,q_1)$ and $\Hc(z,r ; q,q_2)$ for some points
$q_1$ and $q_2$ in $\cl \big(\{ s, x'\} \big)$,

$(ii)$ $\Hc(o,r ; u, x')$ for some point $u$ in $\cl \big ( \{ q, x
\}\big )$.
\end{lem}

\begin{proof}
Since  $\HP (y, x, z; o, q, r, s)$ and $\Hc(y,z ; x,x')$, $\HP (y,
q, o; x', z, s, r)$. Thus, there is a point $q_1$ with $\Hc(y,o;
q,q_1)$ and with $\{q_1,s,x'\}$ collinear, proving part $(i)$.

From the previous paragraph we obtain $\HP (o, q_1, y; r, x', z,
s)$. Hence, by part $(i)$ there is a point $u$ with  $\Hc(o,r ;
x',u)$ and with $\{s,u,q\}$ collinear.
\end{proof}

\section {Projective planes}

In this section we introduce the concept of \emph{harmonic closure}
and state its basic properties (Lemma \ref{closure}). Using this
notion we show that given an embedding of $L_p$ in a harmonic
matroid $M$, the harmonic closure of $L_p $ is a Desarguesian
projective plane of order $p$. This is a constructive geometric
proof. Indeed, using the concept of harmonic conjugation, we extend
each line of the form $\langle j,-1, i \rangle _p$ to get a set in
$M$ with $p+1$ elements. The set formed by the union of all these
extended sets in $M$ and the point $[0,1,0]$ will be the set of
points of a projective plane in $M$. This method seems to work fully
only for prime orders and therefore depends on coordinates and
indeed on primality.

Let $\mathcal {P} (M)$ be the power set of the ground set of a
harmonic matroid $M$. We define $ h \colon \mathcal {P} (M)
\rightarrow  \mathcal {P}(M) $ by $ h(S) = S \cup \{ x \in M \colon
\Hc(a,c ; b, x) \text { for some } a, b, c \in S \}, $ $h^n (S) = h
(h^{n-1} (S))$, and $ h^ \infty (S) = \bigcup _{n\ge 0}h^n(S).$ We
call  $h^ \infty (S)$ the \emph{harmonic closure} of $S$ in $M$.
Note that $h^ \infty $ is an abstract closure operator. If $ h^
\infty (S) = S$, then we say $S$ is \emph{harmonically closed}.

For Lemmas \ref{LPP} to \ref{cp2}  we assume that $p$ is an odd
prime and that $L_p$ is embedded in a harmonic matroid $M$. For $i,j
\in \mathbb Z_p$ define
$$\langle j , -1,i \rangle := h^{\infty}\big( \langle j ,-1, i \rangle_p
\big).$$ This defines a set in $M$ from a set in $L_p$.  The
notation for points in $L_p$ is carried over from the projective
plane, but note that the projective points $[a,b,c]$ that are not in
$L_p$ are not in $M$.  Instead, we will construct points in $M$ with
those same names, by harmonic conjugation, and show that they have
the same incidence structure in $M$ as the points of the same name
in the projective plane. That will justify the use of the projective
names.

\begin{lem} \label{closure} For a set $S$ in a harmonic matroid,
$S \subseteq h^{\infty} (S) \subseteq \cl (S)$ and $\rk (S ) = \rk
(h^{\infty} (S)).$
\end{lem}

\begin{thm} \label{PP}
For each embedding of $L_p$ in a harmonic matroid $M$, $h^
\infty(L_p)$ is a Desarguesian projective plane of order $p$ in $M$.
\end{thm}

The proof for $p \not = 2$ is based on Lemma \ref{LPP} and
Propositions \ref{LPP2} and \ref{LPP3}.
 In Lemmas \ref{LPP}, \ref {cp} and \ref{LPP3}
we show that $\langle x,y,z \rangle$ is, in fact, a collinear set.
Lemma \ref{LPP} is the
hardest part of the proof. It is proved by building lines, repeating
the harmonic conjugation until all the affine points are completely
constructed.

\begin{lem} \label{LPP}
For each pair $s, k$ of elements of $\mathbb Z_p$, there is a unique
point $[1, s,k]$ in $M$ such that:
\begin{align*}
(i) \   & \ \text{ for each } i,j \in \mathbb Z_p, \displaystyle
{\bigcap _{ t \in \mathbb Z_p }
            \langle j + k t, -1, i + t \rangle =  \big \{ [1,j-ki,-k] \big\}, } \\
(ii) \  & \ \text{ for each }i,j, t \in \mathbb Z_p,  \big \{
[0,i+t,1], [1,j+kt,0],[1, j-ki,-k]\big \} \text{ is a collinear set,} \\
(iii)  \ & \ \text{ the points } [1,s,-k], [0,t,1], [0,1,0] \text{ for } s,t \in \mathbb Z_p , \text{ are all different,}\\
(iv) \ & \ \text{ for each } k \in \mathbb Z_p, \big\{ [1,t,-k] : t \in \mathbb Z_p \big\}
            \cup \big\{ [0,1,0] \big\} \text{ is a collinear set.}
\end{align*}
\end{lem}

\begin{proof}
For $k \in \mathbb Z_p$, let $P(k)$ be the statement: for each pair
$ i,j$ of elements of ${\mathbb Z_p}$ there is a point  $[1,
j-ki,-k]$ such that the properties $(i)$ -- $(iv)$ above hold.

Note that $P(-1)$ and $P(0)$ hold by the definition of $L_p$. We
prove $P(k)$ for $1 \leq k \leq p-2$.

We now suppose that for some fixed $k \in \{1, \dots , p-2\}$,
$P(k')$ holds for  $ 1 \leq k' < k$.

Let $s \in \mathbb Z_p$. From $P(k-2)$ $(ii)$  and $P(k-1)$ $(ii)$
both with $t=0$ and $j+s$ instead of $j$, we deduce that
$$
\big \{[0,i,1], [1, j+s-(k-2)i,-(k-2)],  [1, j+s-(k-1)i,-(k-1)]
\big\}
$$
is collinear. Taking $t=s$ and $j+s$ instead of $j$ in $P(k-2)$
$(ii)$, and $t=s$ in $P(k-1)$ $(ii)$ we deduce that
$$
\big \{[0,i+s,1], [1, j+s-(k-2)i,-(k-2)], [1, j-(k-1)i,-(k-1)]
\big\}
$$
is collinear. Similarly taking $t=s$ and $j+2s$ instead of $j$ in
$P(k-2)$ $(ii)$, and $t=s$ and $j+s$ instead of $j$ in $P(k-1)$
$(ii)$, we deduce that
\begin{equation}\label{col2}
\big \{[0,i+s,1], [1, j+2s-(k-2)i,-(k-2)],  [1, j+s-(k-1)i,-(k-1)]
\big\}
\end{equation}
is collinear. These three sets with $s \not = 0$, and $P(0) (iv)$,
$P(k-2) (iv)$ and $P(k) (iv)$ imply that
\begin{multline}\label{eq3}
\HP \big( [0,i+s,1],[1, j-(k-2)i+2s,-(k-2) ], [1, j-(k-1)i +s,-(k-1)]; \\
[0,i,1 ], [0,1,0], [1, j- (k-1) i,-(k-1) ], [1, j-(k-2)i+s,-(k-2)
]\big).
\end{multline}
Thus, there is a point $x_s$ such that
$$ \Hc \big( [0,i+s,1],[1, j-(k-1)i+ s,-(k-1)];  [1, j-(k-2)i+2s,-(k-2) ], x_s \big).$$
By Lemma \ref {tecnico} (\emph{ii}),
$$ \Hc \big( [0,i,1],[1, j-(k-1)i,-(k-1)];  u, x_s \big),$$
where $u$ belongs to $\cl \big([0,i,1],[1j-(k-1)i,-(k-1)]\big)\cap
\cl\big( [0,1,0],[1, j-(k-2)i+2s,-(k-2) ]\big)$. In particular $x_s$ is
independent of $s$. The point $x_s$ is called $[1,j-ki,-k]$. Thus
\begin{equation}\label{eq5}
\big \{ [0,i+t,1], [1, j-(k-2)i+ 2t,-(k-2)], [1, j-(k-1)i+
t,-(k-1)], [1, j-ki,-k] \big \}
\end{equation}
is collinear. This and $P(k-1)$ $(ii)$ imply that
%\begin{equation}\label {ecu8}
$$[1, j-ki,-k ] \in \langle (j+t) + (k-1) t, -1,i+t \rangle .$$
%\end{equation}
\noindent Therefore
\begin{equation}\label{ecu9}
\bigcap_{t \in {\mathbb Z_p}}  \langle j + k t, -1,i+t \rangle =
\big\{ [1, j - k i,-k] \big\},
\end{equation}
thus proving $P(k)$  $(i)$.

By $P(k-1) (ii)$ with $j+t$ instead of $j$, we deduce that
$[1,j+kt,0]$ is collinear with the points in (\ref{eq5}). So,
$P(k)(ii)$ holds.

We now prove $P(k) \ (iii)$. We prove that $[1,j-ki,-k]$ is a new
point. Suppose that this point is not new. Then the point is  either
$[0,i,1]$ or is of the form $[1,j'-si,-s]$ for some $s$ in
$\{-1,0,1, \dots , k-1 \}$.

By the definition of harmonic conjugate $[1, j - k i,-k] \not = [0,i,1]$.

We now suppose that $[1,j-ki,-k] = [1,j'-si,-s]$. First of all,
suppose  that $j= j'$. $P(s) (ii)$ with $t=-i$ and $j'$ instead of
$j$ and  $P(k) (ii)$ with $t=-i$ imply that
$$\big\{[0,0,1], [1,j+k(-i),0], [1,j'+s(-i),0]\big\}$$
is collinear. That is a contradiction because the set is not
collinear in $L_p$.

We suppose that $j \not = j'$. $P(s) (ii)$ with $t=0$ and $j'$
instead of  $j$ and $P(k) (ii)$ with $t=0$  imply that $\big
\{[0,i,1], [1,j',0], [1,j,0] \big\}$ is collinear. That is a
contradiction because the set is not collinear in $L_p$.

Since $[1,j-ki,-k]$ is a new point, $\{ [1,t,-k] : t \in \mathbb Z_p\}$
is a set of new points in $M$. In particular, $ [1,t,-k] \not = [0,1,0]$.

We now prove that the points in the following set are different
$$\big \{ [1,t,-k]: t \in \mathbb Z_p \big \} \cup \big \{ [0,1,0]  \big \}. $$
\noindent Suppose that there are two points $[1,r_1,-k]$ and
$[1,r_2,-k]$ that are equal with $r_1 \not = r_2$. $P(k) \ (ii)$
with $i= 0$, $j=r_1$ and $t=0 $ and $P(k) \ (ii)$ with $i= 0$,
$j=r_2$ and $t=0 $ imply that $\big\{ [0,0,1], [1,r_1,0],[1,r_2,0]
\big\}$ is collinear. That is a contradiction because the set is not
collinear in $L_p$.

We prove $P(k) \ (iv)$. Let $[1,t,-k]$ be a point in $\big \{
[1,t,-k]: t \in \mathbb Z_p \big \} \cup \big \{ [0,1,0]  \big \}$.
We fix $i=0$ and $j=0$ in (\ref{eq3}), and take $s=-t$. By Lemma
\ref{tconjugate},
$$\Hc \big( [0,-t,1],[1,0, -(k-1)] ;[1,-t,-(k-2)], x_t \big)$$
for some point $x_t$ that is collinear with  $[0,1,0]$ and $ [1,
0,-k]$. By taking $i = 0$ and $j= t$ and $-t$ instead of $t$ in
(\ref {eq5}), we obtain that $ x_t = [1,t,-k]$. So, $\big \{
[1,t,-k]: t \in \mathbb Z_p \big \} \cup \big \{ [0,1,0]  \big \}$
is a collinear set. (By uniqueness of the harmonic conjugate we see
that the result does not depend on the choice of $i$ and $j$.) This
completes the proof of $P(k)$.
\end{proof}

We remark that if we repeat the above procedure again to obtain
$P(p-1)$ from $P(p-3)$ and $P(p-2)$, then we obtain the point $[1,
j-(p-1)i,-(p-1)]$ with
\[ \displaystyle { \bigcap _{t \in \mathbb Z_p}  \langle j +(p-1)t, -1,i+t \rangle =
\big\{ [1, j-(p-1)i,-(p-1)]} \big \}.\] By the definition of
$\langle j +(p-1)t, -1,i+t \rangle_p$, the point $[1, j+i,1]$
belongs to $\langle j +(p-1)t, -1,i+t \rangle $. So,
$$\displaystyle { [1, j+i,1] \in \bigcap _{t \in \mathbb Z_p} \langle j +(p-1)t, -1,i+t \rangle }.$$
Thus, $[1, j-(p-1)i,-(p-1)]=
[1, j+i,1]$.

The set $\big\{ [1,t,-k] : t \in \mathbb Z_p \big\} \cup \big\{
[0,1,0] \big\}$  is denoted by $\langle 1, 0, k^{-1} \rangle \text{
if } k \not = 0 $ and by $\langle 0, 0, 1 \rangle \text{ if } k =
0$. From the proof of the previous lemma, part $(iv)$, we deduce
also that $\langle j, -1,i \rangle \cup \langle 1, 0,k^{-1} \rangle$
has rank $3$.

\begin{lem} \label{product}
The point $[x,y,z]$ belongs to  $\langle a,b,c \rangle$ if and only
if $xa +yb+ zc = 0$.
\end{lem}

\begin{proof}
A point $[x,y,z]$ is either $[0,1,0]$, or of the form $[1,l,k]$ or
$[0,l,1]$ for $l,k \in \mathbb Z_p$. A set $\langle a,b,c \rangle$
is either $\langle 0,0,1 \rangle$, or of the form $\langle 1,0,
k^{-1} \rangle$ for $k \not = 0$, or of the form $\langle j,-1,i
\rangle$ for $i,j \in \mathbb Z_p$.

First of all we consider the point $[0,1,0]$. This point does not
belong to a set of the form $\langle j,-1,i \rangle$. Indeed,
suppose that $[0,1,0] \in \langle j,-1,i\rangle$. So, $\big\{
[0,1,0],[0,i,1],[1,j,0] \big \} \subseteq \langle j,-1,i \rangle$.
Therefore by Lemma \ref{closure}
$$\rk \Big( \big\{ [0,1,0],[0,i,1],[1,j,0] \big\} \Big)  \leq \rk \big( \langle j,-1,i \rangle \big) = 2.$$
That is a contradiction because $\big\{ [0,1,0],[0,i,1],[1,j,0]
\big\}$ is not a collinear set of $L_p$. This and Lemma \ref{LPP}
(\emph{iv}) imply that the point belongs only to the sets $\langle
0,0, 1 \rangle$  and $\langle 1,0, k^{-1} \rangle$ for $k \not = 0$.
Thus $[0,1,0] \in \langle a,b,c \rangle$ if and only if $xa +yb+ zc
= 0$.

We suppose that $[x,y,z]$ is a point of the form $[0,l,1]$ for $l
\in \mathbb Z_p$. From Lemma \ref{LPP} parts (\emph{iii}) and
(\emph{v}) we deduce that  $[0,l,1] \not \in \langle 1,0,t \rangle$
for $t \in \{1, \cdots , p-1\}$. By the incidence relation between
points and lines in $L_p$, the point $[0,l,1]$ belongs exactly to
the sets $\langle 1,0,0 \rangle$ and  $\langle t,-1,l \rangle_p$ for
$t \in \mathbb Z_p$. These imply that $[0,l,1]$ belongs exactly to
the sets $\langle 1,0,0 \rangle$ and  $\langle t,-1,l \rangle$ for
$t \in \mathbb Z_p$. Now, for $[x,y,z] = [0,l,1]$ is easy to verify
that $[x,y,z] \in \langle a,b,c \rangle$ if and only if $xa +yb+ zc
= 0$.

We suppose next that $[x,y,z]$ is of the form $[1,l,s]$ for $l,s \in
\mathbb Z_p$. From Lemma \ref{LPP} parts (\emph{iii}) and
(\emph{iv}) we deduce that neither  $\langle 1,0,t \rangle$ for $t
\not = -s^{-1}$ nor $\langle 0,0,1 \rangle$ contain the point
$[1,l,s]$ if $s \not = 0$ and we deduce also that  $[1,l,0] \not \in
\langle 1,0,t \rangle$ for $t \in \{1, \cdots , p-1 \}$.

We now suppose that $[1,l,s]  \in \langle m,-1,n \rangle$. Thus,
$\big\{[1,l,s],[0,n,1],[1,m,0] \big\}$ is collinear. This, taking
$l=j+si$ for some $i,j \in \mathbb Z_p$ and Lemma \ref{LPP}
(\emph{ii}) with $t = n-i$, implies that
$$\big\{[1,m,0],[1,j+si,s],[1,j-st,0],[0,i+t,1] \big\}$$
is collinear. So, $\big\{ [1,m,0], [1,j-st,0],[0,i+t,1] \big\}$ is
collinear in $L_p$. This implies that $[1,m,0] = [1,j-st,0]$.
Therefore $\langle m,-1,n \rangle = \langle j-st,-1,i+t \rangle$.
This completes the proof that $[1,l,s]$ belongs only to $\langle j
-st, -1, i+t \rangle$ for $t \in \mathbb Z_p$ and belongs either to
$\langle 1,0,-s^{-1} \rangle$ if $s \not = 0$ or  $\langle 0,0,1
\rangle$ if $s = 0$. Thus, $[1,l,s] \in \langle a,b,c \rangle$ if
and only if $xa +yb+ zc = 0$.
\end{proof}

For $i,j \in \mathbb Z_p$ we denote by $C(j,i)$ the set
$$C(j,i) := \big\{[0,i,1], [1, j-k i,-k]: k \in \mathbb {Z}_p \big \}.$$

\begin{lem}  \label{cp}  Let $i,j, l,t \in  \mathbb Z_p$. Then

(i) $C(j,i) \subseteq \langle j, -1,i \rangle$

(ii) If $j \not = s$ or $i \not = t$ then  $C(j,i)$ and $C(s,t)$ are
different and they intersect in a point.
\end{lem}

\begin{proof} We only prove one of the two cases of (\emph{ii}).
If $i \not = t$ and $C(j,i)=C(s,t)$, then $[0,t,1] \in C(j,i)$.
Thus, by Lemma (\ref{product}) $0j+t(-1)+1(i)=0$. That is
contradiction because $i \not = t$. By Lemma \ref{LPP} (\emph{i})
and taking $s = j+kl$ and $t=i+l$ for some $k,l \in \mathbb Z_p$, it
follows that $C(j,i)$ and $C(s,t)$ intersect in a point.
\end{proof}

\begin{lem} \label{cp2} If $i,j,k$ are elements of $ \mathbb Z_p$,
then $\langle j,k,i \rangle$ is subset of $h^ \infty (L_p )$.
\end{lem}

\begin{proof}
For each pair $i,j $ of elements of  $\mathbb Z_p,  \ \langle j,-1,i
\rangle _p \subset L_p$
 and $\langle 0,0,1 \rangle \subset L_p$. Then
$h^{\infty} \big (\langle j,-1,i \rangle _p \big) \subset
h^{\infty}(L_p)$ and $h^{\infty}\big( \langle 0,0,1 \rangle \big)
\subset h^{\infty}(L_p)$. To complete the proof, we deduce from the
procedure in the proof of Lemma \ref{LPP} that for each $k \in
\mathbb Z_p$,  $\langle 1,0,k \rangle$ is a subset of
\[
\bigg( \bigcup
_{i,j \in \mathbb Z_p}  \langle j,-1,i \rangle \bigg) \ \cup \langle
0,0,1 \rangle.
 \qedhere
 \]
\end{proof}

For $p$ a prime we define:

$P    \ \  := \ \big \{ [x,y,z]: x, y,z \in \mathbb Z_p, \text{ not all equal }0 \big \},$

\indent \indent \ \ := \ $\big \{ [1,i,k], [0,i,1], [0,1,0] :  i, k \in \mathbb Z_p \big \},$

$\cL _1 \ := \ \big \{  \langle 0,0,1 \rangle, \langle 1,0,k \rangle
: k \in \mathbb Z_p \big\},$

$\cL _2 \ := \ \big\{ C(j,i): i,j \in \mathbb Z_p \big\},$

$\cL    \ \ := \ \cL _1 \cup \cL _2.$

\begin{prop} \label{LPP2}
The sets $P$ and $\cL $ are the sets of points and lines of a
Desarguesian projective plane of order $p$.
\end{prop}

\begin{proof} For $p=2$ the proof follows from the definition of
$L_2$.

Now suppose that $p \not = 2$. First we prove that any two sets in
$\cL $ intersect in a point. By Lemma \ref {cp} (\emph{ii}), the
intersection of any two sets of the form $ C( j,i )$ is exactly a
point. Let $l \in \cL _1$. If $l = \langle 1,0,k^{-1} \rangle$ for
some $k \not = 0$ in $\mathbb Z _p$ then a point in $l\setminus
[0,1,0]$ has the form $[1,s,-k]$ for $s \in \mathbb Z_p$. Taking $s
= j-ki$, it follows that $[1,j-ki,-k] \in C(j,i) \cap l$. A
similarly argument applies if $l$ equals $\langle 1,0,0 \rangle$ or
$ \langle 0,0,1 \rangle $. Finally, by Lemma \ref{LPP}
(\emph{iv}), any two sets in $\cL _1$ intersect at $[0,1,0]$.
Therefore the intersection of any two sets in $\cL$ is a point.

We now prove that given two points in $P$ there exists a set in
$\cL$ containing both points.

Let $v \in P$  distinct from $[0,1,0]$. The construction in the
proof of Lemma \ref{LPP}, and the definition of $L_p$, imply that
there are $p$ sets of the type $C(j,i)$ and one set that is either
$\langle 0,0,1 \rangle $ or is of the form $\langle 1,0,k \rangle$,
such that all these $p+1$ sets intersect in $v$, and that $P$ is
contained in the union of those $p+1$ sets. Therefore, given $s$ in
$P$, there is a set in $\cL$ that contains both points $v$ and $s$.

For $v = [0,1,0]$, the analysis is as above but taking the sets
$\langle 0,0,1 \rangle $, $\langle 1,0,k \rangle$ for $k \in \mathbb
Z_p$.

By Lemma \ref {product} it follows that $\Pi$ is Desarguesian.
\end{proof}

We denote by $\Pi$ the projective plane given in Proposition \ref{LPP2}.
This proposition shows that we have constructed a projective
plane within $M$, but we still need to prove that $\Pi$ is the
harmonic closure of the initial matroid $L_p$.

\begin{prop} \label{LPP3} Assume that $L_p$ is embedded in a harmonic matroid
$M$. Then $P$ and all lines in $\Pi$ are harmonically closed in $M$.
In fact $P = h^{\infty}(L_p)$ and for any pair $i,j$ of elements of
$\mathbb Z_p$, $C (j,i) = \langle j,-1,i \rangle$.
\end{prop}

\begin{proof} If $p=2$ then $L_2= \Pi$, so there is nothing to
prove.

First of all we prove that every line in $\Pi$ is harmonically
closed. Let $a, b, c$ be distinct points in a line $l\in \Pi$. Let
$t$ be a point in $\Pi$ that is not in $l$. It is easy to show that
four points in a projective plane of order $p$, where exactly three
are collinear give rise to a matroid as in Figure \ref{f1}
(\emph{b}). Thus, there are points $x,y,w \in P$ with $\HP
(a,b,c;x,t,y,w )$ in $P$, therefore in $M$. So there is a point $v
\in M$ with $\Hc \big (a,c ;b,v \big)$ and with $ \cl \big( \{a,,c
\} \big ) \cap \cl \big ( \{x,y \} \big) = \{ v \} $. As $v$ is the
unique point of $M$ in this intersection, it follows that $v\in
C(j,i)$ because $\Pi$ is a projective plane.

We now prove that $P$ is harmonically closed. By Lemma \ref{LPP} and
Proposition \ref{LPP2}, every three collinear points in $P$ belong
to a set $\langle j,k,i \rangle $. Since this set is harmonically
closed, the conclusion follows.
\end{proof}

The proof of Theorem \ref{PP} follows from Propositions \ref{LPP2}
and \ref{LPP3}.

\begin{prop} Assume that $L_p$ is embedded in a harmonic matroid $M$. Then
$$\langle j,-1,i \rangle = \cl \big(\langle j,-1,i
\rangle_p\big) \cap  h^ \infty \big( L_p \big).$$
\end{prop}

\begin{proof} By Lemma \ref{closure},
$\langle j,-1,i \rangle =  h^{\infty} \big( \langle j,-1,i \rangle_p
\big) \subseteq  \cl \big ( \langle j,-1,i \rangle_p \big )$. By
Lemma \ref{cp2}, $\langle j,-1,i \rangle \subseteq h^ \infty \big(
L_p \big )$. These imply that $\langle j,-1,i \rangle \subseteq \cl
\big ( \langle j,-1,i \rangle_p \big) \cap  h^ \infty \big( L_p
\big)$.

We now prove the reverse containment. Since $L_p \subseteq P $,
$h^{\infty} ( L_p ) \subseteq h^{\infty} (P) = P $. Let $x \in \cl
(\langle j,-1,i \rangle_p) \cap  h^ \infty ( L_p )$. Thus,
$\big\{x,[0,i,1], [1,j,0] \big\}$ is a collinear subset of $h^
\infty ( L_p )$. Hence, by Proposition \ref{LPP2}, $\big\{x,[0,i,1],
[1,j,0] \big\} \subseteq C(j,i)$. So, by Proposition \ref{LPP3}, $x
\in \langle j,-1,i \rangle$.
\end{proof}

\begin{cor} \label{Pfam}  Let $F$ be a field of characteristic $p$.
If  $A(K/F)$ is a full algebraic matroid of rank at least three,
then there exists an embedding of $L_p $ in $A(K/F)$, and for every
such embedding $h^ \infty (L_p )$ is a Desarguesian projective plane
of order $p$ in $A(K/F)$.
\end{cor}

\begin{proof} By {\cite [Theorem 8]{rf}} the matroid $L_p $
embeds in $A(K/F)$. Since a full algebraic matroid is a harmonic
matroid, apply Theorem \ref{PP}.
\end{proof}

\section {A minimal matroid } \label{minimalmatroid}

As a natural question we ask: what is a minimal matroid whose
harmonic closure is a projective plane? In \cite {rf} Fl\'orez
proved that $M(p):= L_p  \setminus \big\{ [1,2,0], \dots , [1,p-1,0]
\big\}$ can be extended by harmonic conjugation to $L_p $. The
matroid $M(p)$ is also known as $R_{\text{cycle}}[p]$, which is the
\emph{Reid cycle matroid} \cite [page 52] {jk}. Here we prove that
$R_{\text{cycle}} [p]$ is a minimal matroid whose harmonic closure
is a projective plane. Note that the ground set of
$R_{\text{cycle}}[p]$ is
$$ E := \big\{ [0,i,1],[1,i,1]: i \in \mathbb Z_p \big \} \cup \big \{ [1,0,0],[1,1,0], [0,1,0]  \big \}.$$

\begin{lem} [See \cite {rf}] \label{mnembedding} Any embedding of
$R_{\text{cycle}}[p]$ in a harmonic matroid $M$ extends uniquely to
an embedding of $L_p $ in $M$.
\end{lem}

\begin{thm} \label{un} Let $p$ be a prime number. Suppose that
$L_p $ is embedded in a harmonic matroid $M$. Then
$R_{\text{cycle}}[p]$ is a minimal submatroid of $L_p$ for which $h^
\infty (R_{\text{cycle}}[p])$ is equal to $h^ \infty (L_p )$.
\end{thm}

\begin{proof} By Lemma \ref{mnembedding} and Theorem \ref{PP},
$h^ \infty (R_{\text{cycle}}[p])$ is equal to $h^ \infty (L_p )$,
which is a projective plane. Note that every proper minor of
$R_{\text{cycle}}[p]$ embeds in a full algebraic matroid over the
rational numbers $\mathbb Q$ \cite [Proposition 12]{ga}.

We prove that  $h^ \infty (R_{\text{cycle}}[p] \setminus x )  \not =
h^ \infty (L_p)$ for $x\in E$. Indeed, we suppose that there is a
point $x \in E$ such that $h^ \infty (R_{\text{cycle}}[p] \setminus
x ) = h^ \infty (L_p)$ in $M$. Thus, $h^ \infty (R_{\text{cycle}}[p]
\setminus x)$ is a projective plane $\Pi$ of order $p$. Since
$R_{\text{cycle}}[p]\setminus x$ embeds in a full algebraic matroid
over $\mathbb Q$, the projective plane $\Pi$ embeds in a full
algebraic matroid over $\mathbb Q$. Since $L_p$ is a submatroid of
$\Pi$, the matroid $L_p$ embeds in a full algebraic matroid over
$\mathbb Q$. That is a contradiction because $L_p$ is algebraically
representable only over fields of characteristic $p$ \cite [Theorem
7]{rf}.
\end{proof}

\section {M\"obius harmonic nets} \label{Mobius}

Veblen and Young \cite{young} introduced the concept of a M\"obius
harmonic net into projective geometry to construct the field of
rational numbers and the finite field $\mathbb{Z}_p$. We generalize
the concept of conjugate sequence (see \cite {coxeter} or \cite
{young}) from projective geometry to a harmonic matroid $M$. We
prove without using the axioms of projective geometry that this
sequence gives rise to a finite field within $M$ and that the idea
of a M\"obius harmonic net generalizes to $M$. In fact we prove that
if $h^ \infty (F_7^-)$ has a finite number of points, then it is a
projective plane of prime order.

A \emph{conjugate sequence} is a sequence of points $A,a_0, \dots ,
a_n, \dots $ in a harmonic matroid with $\Hc(A,a_{i};a_{i-1},
a_{i+1})$ for $i \in \mathbb Z_{\geq 0}$. A conjugate sequence is
\emph{modular} if at least one of its points is equal to one of the
preceding points. A modular sequence where $a_n$ is the first point
that is equal to one of its preceding points is denoted $A,a_0, a_1,
\dots , a_{n-1}$. A sequence $S$ of collinear points is a
\emph{M\"obius harmonic net} or \emph{net of rationality} if $h^
\infty (S)= S$.

The matroid $R_{\text{cycle}}[p]$ defined in the previous section
generalizes to $R_{\text{cycle}}[n]$ for $n \ge 2$.
$R_{\text{cycle}}[n]$ is a simple matroid of rank 3. We need only
specify the lines of size of at least three, because any set of two
points not belonging to the same dependent line is also a line. The
points are denoted by: $a_0,b_0,a_1,b_1,\dots , a_{n-1},b_{n-1},
c_0,c_1,d.$ The dependent lines are:
$$ \{ a_0,a_1, \dots , a_{n-1}, d \},\{ b_0,b_1,
\dots ,b_{n-1}, d \}, \{ c_0,c_1,d \},\{ a_i,b_{i},c_0 \},\{
a_i,b_{i+1},c_1 \},$$ for $i=0,1, \dots, n-1$,  where  $b_n = b_0$.

\begin{prop} \label{pro1} If $d, b_0, \dots , b_{n-1}$ is a modular
sequence in a harmonic matroid $M$, then $n$ is a prime number and
there are points  $c_1, c_0, a_1, a_0$ in $M$ such that $h^ \infty
\big (\{d, b_0, b_1, c_1, c_0, a_1, a_0\} \big)$ is a projective
plane of order $n$.
\end{prop}

\begin{proof} Since $\Hc(d,b_1;b_0,b_2)$, there are points
$c_1, c_0, a_1, a_0$ with $\HP(d, b_0, b_1; c_1, c_0, a_1, a_0)$ and
$\{c_1,a_1,b_2\}$ collinear. Note that by Lemma \ref{tconjugate}
there is a point $a_2$ such that $ \{a_2, b_2,  c_0\}$ and $\{d,
a_0, a_1,  a_2\}$ are collinear sets.

For $t \in \mathbb{Z}_n $, let  $P(t)$ be the statement: There is a
point $a_t$ in $M$ such that $\{a_t,b_t, c_0\}$,
$\{a_t,b_{t+1},c_1\}$ and $\{d,a_0, \dots, a_t\}$ are collinear
sets.

From the first paragraph we can see that $P(0)$ and $P(1)$ hold.

We now suppose that $P(t-2)$ and $P(t-1)$ hold for some fixed $t \in
\{2, \dots , n-1 \}$, and we deduce $P(t)$.

From  the modular sequence $d, b_0, \dots , b_{n-1}$, $\HP(d, b_0,
b_1; c_1, c_0, a_1, a_0)$, $P(t-2)$, and $P(t-1)$ we deduce that
$$\HP( d, b_{t-2}, b_{t-1}; c_1, c_0, a_{t-1}, a_{t-2}).$$
Thus, there is a point $x$ with $\Hc(d,b_{t-1};b_{t-2},x)$ and such that
$\{c_1,a_{t-1},x\}$ is a collinear set. Since
$\Hc(d,b_{t-1};b_{t-2},b_t)$, by uniqueness of the harmonic
conjugate $x = b_t$. Therefore by Lemma \ref{tconjugate} there is a
point $a_t$ with $ \{a_t, b_{t}, c_0\}$ and $\{d, a_{t-2}, a_{t-1},
a_t\}$ collinear sets. These imply that
$$\HP( d, b_{t-1}, b_t; c_1, c_0, a_t, a_{t-1}).$$
Thus, there is a point $b_{t+1}$ such that
$\Hc(d,b_t;b_{t-1},b_{t+1})$ and $\{c_1,a_t,b_{t+1}\}$ is a
collinear set. (Note that if $t = n-1$ then $\{a_{n-1},b_n,c_1\}$ is
collinear and $b_n \in \{d, b_0, \dots , b_{n-1} \}$.) That proves
$P(t)$.

We now prove that $h^ \infty \big (\HP(d, b_0, b_1; c_1, c_0, a_1, a_0) \big)$
is a projective plane of prime order.

We suppose that $b_n = b_t$ for some $t \in \{1, \dots , n-1 \}$.
Since $\{c_1,a_{t-1},b_{t}\}$ and  $\{c_1,a_{n-1},b_n \}$ are
collinear, the set $\{c_1,a_{t-1}, a_{n-1} \}$ is collinear. This
implies that $\{c_1,a_0, a_1 \}$ is collinear. That is contradiction
because the set is not collinear in $\HP(d, b_0, b_1; c_1, c_0, a_1,
a_0)$. This implies that $b_n = b_0$.

Note that $M|E$ where $E= \{d,a_0, \dots , a_{n-1}, b_0, \dots ,
b_{n-1}, c_0, c_1 \}$ is the matroid $R_{\text{cycle}}[n]$. Thus
$R_{\text{cycle}}[n]$ is embedded in a harmonic matroid.  By \cite
[Theorem 10] {rf}, $n$ is a prime number.

By definition of $E$ we deduce that
$$ h^ \infty ( \{d, b_0, b_1; c_1, c_0, a_1, a_0\}) = h^ \infty (M|E) =  h^\infty (R_{\text{cycle}}[n]).$$
Therefore, by Proposition \ref{un}   $h^ \infty
(R_{\text{cycle}}[n])$ is a projective plane of order $n$.
\end{proof}

\begin{cor}If $d, b_0, \dots , b_n$ is a modular sequence in a harmonic
matroid $M$, then it is a M\"obius harmonic net.

\end{cor}

\section* {Remarks}

\emph{Question.} In \cite {b2} can be found the definition of the
complete lift matroid of a group expansion graph, $L_0(\mathfrak{G}
K_3)$, where $\mathfrak{G}$ is a group. Our $L_p$ is
$L_0(\mathbb{Z}_p K_3)$. In \cite {rf} Fl\'orez proved that
$L_0(\mathbb{Z}_n K_3)$ is not embeddable in any harmonic matroid if
$n$ is not a prime number. It will be interesting to know whether,
if $L_0(\mathbb{Z} K_3)$ embeds in a harmonic matroid $M$, then $h^
\infty (L_0(\mathbb{Z} K_3))$ is a projective plane in $M$. It is
known that $L_0(\mathbb{Z} K_3)$ embeds in a full algebraic matroid
over the rational numbers and it is also known that this matroid
embeds in a Desarguesian plane coordinatized over the rational
numbers.

Our construction does not produce a projective plane except for a
prime order. In fact, if $\mathfrak{G} = (\mathbb Z_p)^n$ where $n
\geq 2$, then $L_0(\mathfrak{G} K_3)$ embeds in a harmonic matroid
(in fact in a projective plane of order $p^n$) but $h^{\infty}\big(
L_0(\mathfrak{G} K_3) \big)$ is not a projective plane. (The proof
will appear elsewhere.) It follows that our method requires prime
coordinates. So, it cannot be done in a manner without coordinates.

\section* {Acknowledgment}
The author is indebted to his advisor, Thomas Zaslavsky, and referees
for their extensive comments and corrections that helped to improve
the presentation.

\end{document}